\documentclass[11pt,a4paper]{article}
\usepackage{indentfirst,mathrsfs}
\usepackage{amsfonts,amsmath,amssymb,amsthm}
\usepackage{latexsym,amscd,multirow}
\usepackage{amsbsy}
\usepackage{color}
\newtheorem{thm}{Theorem}
\newtheorem{lem}{Lemma}
\newtheorem{cor}{Corollary}

\newtheorem{prop}{Proposition}

\newcounter{alg}
\newlength{\lefttab}
\newlength{\numberoffset}
{\trivlist
   \topsep=0pt\parsep=0pt\itemsep=0pt
   \addtolength{\lefttab}{1.25em}
   \addtolength{\numberoffset}{1.25em}
   \leftskip=\lefttab}%
  {\endtrivlist}

\begin{document}
\title{A Class of Binomial Permutation Polynomials}
\author{Ziran Tu
\thanks{Z. Tu is with the School of Mathematics and Statistics, Henan University of Science and
Technology, Luoyang 471003, China. Email: naturetu@gmail.com.},
Xiangyong Zeng
\thanks{X. Zeng is with the Faculty of Mathematics and Computer Science, Hubei
University, Wuhan 430062,
 China. Email:
xyzeng@hubu.edu.cn.}, Lei Hu
\thanks{L. Hu is with the State Key Laboratory of Information Security,
Institute of Information Engineering, Chinese Academy of Sciences, Beijing {\rm 100093},
China. Email:
hu@is.ac.cn.}, Chunlei Li
\thanks{C. Li is with the Department of Informatics,
University of Bergen, Bergen {\rm N-5020}, Norway. Email: chunlei.li@ii.uib.no}}
\date{}
\maketitle
\begin{quote}
{\small {\bf Abstract:} In this note, a criterion for a class of binomials to be permutation polynomials is proposed.
As a consequence, many classes of binomial permutation polynomials and monomial complete permutation polynomials are obtained.
The exponents in these monomials are of Niho type.}


\end{quote}

\section{Introduction}
For a prime power $q$, let $\mathbb{F}_{q}$ be the finite field with $q$ elements and $\mathbb{F}_{q}^*$ denote its multiplicative group.
  A polynomial $f\in \mathbb{F}_{q}[x]$
is called {\it a permutation polynomial} (PP) if its associated polynomial
mapping $f$: $c\mapsto f(c)$ from $\mathbb{F}_{q}$ to itself is a
bijection.  A permutation polynomial $f(x)$ is referred to as {\it a complete permutation polynomial} (CPP)
if $f(x)+x$ is also a permutation over $\mathbb{F}_{q}$ \cite{Niederreiter}.
Permutation polynomials were studied first by Hermite in \cite{hermite} for
the case of finite prime fields and by Dickson in \cite{dickson} for arbitrary finite fields.
Permutation polynomials have been intensively studied in recent years for their applications in
cryptography, coding theory and
combinatorial design theory. Complete  classification  of permutation polynomials
seems to be a very challenging problem. However,
a lot of interesting results on permutation polynomials were found through
studying monomials, binomials,  quadratic polynomials and some polynomials with special forms.

In this note, for an even integer $n$ and an integer $t\geq 2$, we investigate the polynomials of the form
\begin{equation}\label{form1}
f(x)=\sum_{i=1}^{t}u_{i}x^{d_{i}},
\end{equation}
where for each $i$ with $1\leq i\leq t$, the element $u_{i}\in \mathbb{F}_{2^n}$ and
$d_{i}\equiv e \,\left({\rm mod}\, 2^{\frac{n}{2}}-1\right)$ for a positive integer $e$. In the case of ${\rm gcd}(d_{1},2^n-1)=1$ and $t=2$,
a criterion for the binomials having the form as in (\ref{form1}) to be permutation polynomials is proposed.
As a consequence, a number of classes of binomial permutation polynomials
and monomial complete permutation polynomials are obtained. The proofs of our
main results are based on a criterion for permutation polynomials in terms of additive characters of the underlying
finite fields \cite{Lidl}.

The remainder of this paper is organized as follows. In Section 2,
we introduce some basic concepts and related results. In Section 3,
a criterion for a class of binomials to be permutation polynomials is given.
Several classes of binomial permutation polynomials and monomial complete permutation polynomials are obtained.
A conjecture on two classes of trinomials is presented in Section 4.

\section{Preliminaries}
For two positive integers $m$ and $n$ with $m\,|\,n$, we use ${\rm
Tr}_{m}^{n}(\cdot)$ to denote the {\it trace function} from
$\mathbb{F}_{2^n}$ to $\mathbb{F}_{2^m}$, i.e.,
$${\rm Tr}_{m}^{n}(x)=x+x^{2^m}+x^{2^{2m}}+\cdot\cdot\cdot+x^{2^{(n/m-1)m}}.$$
For a Boolean function $f$:
$\mathbb{F}_{2^n}\rightarrow\mathbb{F}_{2}$, the {\it Walsh
spectrum} of $f$ at $a\in \mathbb{F}_{2^n}$ is defined by
$$W_f(a)=\sum\limits_{x\in \mathbb{F}_{2^n}}(-1)^{f(x)+{\rm Tr}^n_1(ax)}.$$

A criterion for PPs can be given by using additive characters of the underlying finite field \cite{Lidl}.
This can also be characterized by judging whether the Walsh spectra of some Boolean functions at $0$ are equal to zero.
\begin{lem} (\cite{Lidl}) \label{lem1} A mapping $g$: $\mathbb{F}_{2^n}\rightarrow
\mathbb{F}_{2^n}$  is a permutation polynomial of $\mathbb{F}_{2^n}$
if and only if for every nonzero $\gamma\in \mathbb{F}_{2^{n}}$,
 $$\sum_{x\in \mathbb{F}_{2^{n}}}(-1)^{{\rm Tr}_{1}^{n}(\gamma g(x))}=0.$$
\end{lem}

For a polynomial $f(x)$ having the form as in (\ref{form1}), without loss of generality, we can assume
that $u_1=1$ and $d_i=s_i\left(2^{\frac{n}{2}}-1\right)+e$ for integers $s_i$ and $e$,
where $1\leq i\leq t$. In the case of ${\rm gcd}(d_{1},2^n-1)=1$,
every nonzero $\gamma\in \mathbb{F}_{2^n}$ can be represented as $\delta^{d_{1}}$ for a unique nonzero $\delta\in \mathbb{F}_{2^n}$. Then
\begin{eqnarray*}
% \nonumber to remove numbering (before each equation)
  \sum_{x\in \mathbb{F}_{2^n}}(-1)^{{\rm Tr}_{1}^{n}(\gamma f(x))}
  &=&\sum\limits_{x\in \mathbb{F}_{2^n}}(-1)^{{\rm Tr}_{1}^{n}\left(\gamma\left(\sum\limits_{i=1}^{t}u_{i}x^{d_{i}}\right)\right)} \\
  &=& \sum_{x\in \mathbb{F}_{2^n}}(-1)^{{\rm Tr}_{1}^{n}\left((\delta x)^{d_{1}}+\sum\limits_{i=2}^{t}u_{i}\delta^{d_{1}-d_{i}}(\delta x)^{d_{i}}\right)} \\
   &=&  \sum_{x\in \mathbb{F}_{2^n}}(-1)^{{\rm Tr}_{1}^{n}\left(x^{d_{1}}+\sum\limits_{i=2}^{t}u_{i}\delta^{d_{1}-d_{i}}x^{d_{i}}\right)}.
   \end{eqnarray*}
By Lemma \ref{lem1}, a  criterion for $f(x)$ to be a PP is given as follows.

\begin{cor}\label{cor1} Let $f(x)=x_1^{d_1}+\sum\limits_{i=2}^{t}u_{i}x^{d_{i}}$ be defined as in (\ref{form1}) and ${\rm gcd}(d_{1},2^n-1)=1$. The polynomial $f$ is a permutation polynomial over  $\mathbb{F}_{2^n}$ if and only if for every $\delta\in \mathbb{F}_{2^n}^{*}$,
\begin{equation}\label{eq2}
    \sum_{x\in \mathbb{F}_{2^n}}(-1)^{{\rm Tr}_{1}^{n}\left(x^{d_{1}}+\sum\limits_{i=2}^{t}u_{i}\delta^{d_{1}-d_{i}}x^{d_{i}}\right)}=0.
\end{equation}
\end{cor}

 If there is a positive integer $e$ such that the exponent $d_{i}\equiv e \,\left({\rm mod}\, 2^{\frac{n}{2}}-1\right)$ for each $i$ with $1\leq i\leq t$, the exponential sum $\sum\limits_{x\in \mathbb{F}_{2^n}}(-1)^{{\rm Tr}_{1}^{n}\left(x^{d_{1}}+\sum\limits_{i=2}^{t}u_{i}\delta^{d_{1}-d_{i}}x^{d_{i}}\right)}$ can be analyzed by a technique used in \cite{Niho}. To this end, we recall the concept of {\it unit circle} of $\mathbb{F}_{2^{2m}}$, which is the set
\begin{equation}\label{set}U=\left\{\lambda\in \mathbb{F}_{2^{2m}}:\lambda^{2^m+1}=1\right\}.\end{equation}

\begin{lem}\label{lem2}
Let $n=2m$ for a positive integer $m$, ${\rm gcd}(d_{1},2^n-1)=1$ and
$d_{i}=s_{i}(2^m-1)+e$ for $i=1,\cdots,t$. For $w_2, \cdots, w_t\in\mathbb{F}_{2^n}$,
the exponential sum
$$\sum_{x\in \mathbb{F}_{2^n}}(-1)^{{\rm Tr}_{1}^{n}\left(x^{d_{1}}+\sum\limits_{i=2}^{t}w_ix^{d_{i}}\right)}=(N(w_{2},\cdots,w_{t})-1)\cdot 2^m,
$$
where $N(w_{2},\cdots,w_{t})$ is the number of $\lambda's$ in $U$ such that
\begin{equation}\label{eq1}
    \lambda^{d_{1}}+\sum\limits_{i=2}^{t}w_i\lambda^{d_{i}}+\left(\lambda^{d_{1}}+\sum\limits_{i=2}^{t}w_i\lambda^{d_{i}}\right)^{2^m}=0.
\end{equation}

\end{lem}
\begin{proof}
Using the polar coordinate representation \cite{Niho}, every nonzero $x\in \mathbb{F}_{2^n}$ can be uniquely  represented as $x=\lambda y$, where $\lambda \in U$ and $y\in \mathbb{F}_{2^m}^{*}$.  Thus,
\begin{eqnarray*}
% \nonumber to remove numbering (before each equation)
  &&\sum_{x\in \mathbb{F}_{2^n}}(-1)^{{\rm Tr}_{1}^{n}\left(x^{d_{1}}+\sum\limits_{i=2}^{t}w_ix^{d_{i}}\right)}\\
   &=& 1+ \sum_{x\in \mathbb{F}_{2^n}^{*}}(-1)^{{\rm Tr}_{1}^{n}\left(x^{d_{1}}+\sum\limits_{i=2}^{t}w_ix^{d_{i}}\right)} \\
   &=& 1+\sum_{\lambda\in U}\sum_{y\in \mathbb{F}_{2^m}^{*}}(-1)^{{\rm Tr}_{1}^{n}\left((\lambda y)^{d_{1}}+\sum\limits_{i=2}^{t}w_i(\lambda y)^{d_{i}}\right)} \\
   &=& -2^m+\sum_{\lambda\in U}\sum_{y\in \mathbb{F}_{2^m}}(-1)^{{\rm Tr}_{1}^{n}\left(\lambda^{d_{1}} y^{e}+\sum\limits_{i=2}^{t}w_i\lambda^{d_{i}} y^{e}\right)}  \\
   &=& -2^m+\sum_{\lambda\in U}\sum_{y\in \mathbb{F}_{2^m}}(-1)^{{\rm Tr}_{1}^{m}\left({\rm Tr}_{m}^{n}\left(\lambda^{d_{1}} y^{e}+\sum\limits_{i=2}^{t}w_i\lambda^{d_{i}} y^{e}\right)\right)}  \\
   &=& -2^m+\sum_{\lambda\in U}\sum_{y\in \mathbb{F}_{2^m}}(-1)^{{\rm Tr}_{1}^{m}\left(\left(\lambda^{d_{1}}+\sum\limits_{i=2}^{t}w_i\lambda^{d_{i}}+
   \left(\lambda^{d_{1}}+\sum\limits_{i=2}^{t}w_i\lambda^{d_{i}}\right)^{2^m}\right)y^{e}\right)}.
   \end{eqnarray*}
Note that ${\rm gcd}(d_{1},2^n-1)=1$ implies ${\rm gcd}(e,2^m-1)=1$. Consequently, $y^e$ runs through all elements in $\mathbb{F}_{2^m}$ as $y$ runs through all elements in  $\mathbb{F}_{2^m}$. Thus,
   \begin{eqnarray*}
   &&\sum_{x\in \mathbb{F}_{2^n}}(-1)^{{\rm Tr}_{1}^{n}\left(x^{d_{1}}+\sum\limits_{i=2}^{t}w_ix^{d_{i}}\right)}\\
   &=& -2^m+\sum_{\lambda\in U}\sum_{y\in \mathbb{F}_{2^m}}(-1)^{{\rm Tr}_{1}^{m}\left(\left(\lambda^{d_{1}}+\sum\limits_{i=2}^{t}w_i\lambda^{d_{i}}+\left(\lambda^{d_{1}}+
   \sum\limits_{i=2}^{t}w_i\lambda^{d_{i}}\right)^{2^m}\right)y\right)} \\
   &=& (N(w_{2},\cdots,w_{t})-1)\cdot 2^m.
\end{eqnarray*}
Then the proof is completed.
\end{proof}

In the next section, we will apply Corollary \ref{cor1}, Lemmas \ref{lem1} and \ref{lem2} to discuss the polynomials with a form as in (\ref{form1}) for $t=2$.

\section{Binomial Permutation Polynomials}

In this section, we investigate the permutation behavior of the binomials having the form as in (\ref{form1}).

\begin{thm}\label{thm1}
Let positive integers $n$, $m$, $s$, $l$, $e$, $d_1$ and $d_2$ satisfy  $n=2m$, $d_{1}=s(2^m-1)+e$, $d_{2}=(s-l)(2^m-1)+e$ and  ${\rm gcd}(d_{1},2^n-1)=1$. Then
the binomial $$x^{d_{1}}+ux^{d_{2}}$$ is a permutation polynomial over $\mathbb{F}_{2^n}$
if the following conditions are satisfied:

(i) $r:={\rm gcd}(l,2^m+1)>1$;

(ii) ${\rm gcd}(e+l-2s,2^m+1)=1$; and

 (iii) $u\in U\setminus U^r$, where the set $U$ is defined in  (\ref{set}) and $U^r=\{v^r: v\in U\}$.
\end{thm}

\begin{proof}
By  Lemma \ref{lem1} and Corollary \ref{cor1}, it suffices to prove that for each $\delta\in\mathbb{F}_{2^n}^*$, the exponential sum equality
\begin{equation}\label{eq2}
   \sum_{x\in \mathbb{F}_{2^n}}(-1)^{{\rm Tr}_{1}^{n}(x^{d_{1}}+u\delta^{d_1-d_2}x^{d_{2}})}=0
\end{equation}
holds for all $u\in U\setminus U^r$.  By Lemma \ref{lem2}, (\ref{eq2}) holds if and only if the equation
\begin{equation}\label{eq3}
    \lambda^{d_{1}}+w\lambda^{d_{2}}+(\lambda^{d_{1}}+w\lambda^{d_{2}})^{2^m}=0
\end{equation}
 on $\lambda\in U$ has a unique solution in $U$ for any $\delta\in \mathbb{F}_{2^n}^*$, where $w=u\delta^{d_1-d_2}$.

Since $\left(\delta^{d_1-d_2}\right)^{2^m+1}=\left(\delta^{l(2^m-1)}\right)^{2^m+1}=1$,
we have $\delta^{d_1-d_2}\in U$. By Condition (iii),  the element $u$ belongs to the set $U$, and then $w=u\delta^{d_1-d_2}\in U$. Thus,
 Equation (\ref{eq3})  is equivalent to
\begin{equation*}
  \lambda^{d_{1}}+w\lambda^{d_{2}}+\lambda^{-d_{1}}+w^{-1}\lambda^{-d_{2}}=0,
\end{equation*}
 which implies
 \begin{equation*}
  w\lambda^{2d_{1}+d_2}+w^2\lambda^{d_1+2d_{2}}+w\lambda^{d_{2}}+\lambda^{d_{1}}=0,
\end{equation*} and then
\begin{equation}\label{eq6}
  (w\lambda^{d_{2}}+\lambda^{d_{1}})(w\lambda^{d_{1}+d_{2}}+1)=0.
\end{equation}

If $w\lambda^{d_{2}}+\lambda^{d_{1}}=0$, then $u(\delta\lambda^{-1})^{d_1-d_2}=1$.
By Condition (i) and that $d_{1}-d_{2}=l(2^m-1)$, we have $\left(u(\delta\lambda^{-1})^{d_1-d_2}\right)^{\frac{2^m+1}{r}}=1$
and $u^{\frac{2^m+1}{r}}=1$, which means $u\in U^r$ since $U$ is a cyclic group of order $2^m+1$.
%Let $\xi\in U$ and its order is $2^{m}+1$. Then $u=\xi^j$ for a positive integer $j$ due to $u\in U$. So, we have  $\xi^{j\frac{2^m+1}{r}}=1$ and then  $2^{m}+1\,|\,j\frac{2^m+1}{r}$, i.e., $r\,|\,j$.
%Thus, $w\lambda^{d_{2}}+\lambda^{d_{1}}=0$ implies $u=v^r$ for some $v$ in $U$. That is to say
Thus, if $u\in U\setminus U^r$, then $w\lambda^{d_{2}}+\lambda^{d_{1}}\neq 0$.
As a consequence, by Equation (\ref{eq6}), we have  $w\lambda^{d_{1}+d_{2}}+1=0$. Condition (ii) shows
$$\begin{array}{rcl}{\rm gcd}(d_{1}+d_{2},2^m+1)&=&{\rm gcd}\big((2s-l)(2^m-1)+2e,2^m+1\big)\\
&=&{\rm gcd}\big(-2(2s-l)+2e,2^m+1\big)\\
&=&{\rm gcd}\big(e+l-2s,2^m+1\big)\\&=&1.\end{array}$$ Therefore, $\lambda=w^{-\frac{1}{d_{1}+d_{2}}}$ is the unique solution of Equation (\ref{eq6}). Thus, Equation (\ref{eq3}) has a unique solution and then (\ref{eq2}) holds.

It follows from Lemma \ref{lem1} and Corollary \ref{cor1} that $x^{d_{1}}+ux^{d_{2}}$ is a permutation polynomial over $\mathbb{F}_{2^n}$.\end{proof}

By Theorem \ref{thm1}, binomial permutation polynomials will be obtained if the conditions in Theorem \ref{thm1} are satisfied.
In the sequel, six classes of binomial permutation polynomials are given by applying Theorem \ref{thm1} and the following simple lemma from elementary number theory.

\begin{lem}\label{lem4} Let $r$ and $s$ be two positive integers. Then,
\begin{itemize}
  \item[(i)] $\gcd(2^r-1, 2^s-1)=2^{\gcd(r,s)}-1$;
  \item[(ii)] $\gcd(2^r-1, 2^s+1)=
  \begin{cases}
    1, & \text{if $r/\gcd(r,s)$ is odd}, \\
    2^{\gcd(r,s)}+1, & \text{if $r/\gcd(r,s)$ is even};
  \end{cases}
  $
  \item[(iii)] $\gcd(2^r+1,2^s+1)=
  \begin{cases}
    2^{\gcd(r,s)}+1, & \text{if both $r/\gcd(r,s)$ and $s/\gcd(r,s)$ is odd}, \\
    1, & \text{otherwise}.
  \end{cases}
  $
\end{itemize}
\end{lem}

\begin{prop}\label{prop1}
Let $n=2m$ for an odd integer $m$ and $k_1, k_2, k_3$ be three nonnegative integers.
Let $d_1=s(2^m-1)+e$ and $d_2=(s-l)(2^m-1)+e$. For a non-cubic element $u\in U$, the polynomial
$$x^{d_1}+ux^{d_2}$$
is a permutation polynomial over $\mathbb{F}_{2^n}$ if one of the followings holds:
\begin{enumerate}
  \item  $e=2^{k_3}+1$, $s=2^{k_3-1}-2^{k_2-1}+1$ and $l=2^{k_1}+1$, where  $k_1$, $k_2$  are odd and $\gcd(k_1,m)=1$;
  \item  $e=2^{k_3}+1$, $s=2^{k_3-1}-2^{k_2-1}$ and $l=2^{k_1}-1$, where $k_1$, $k_2$ are even and $\gcd(k_1,m)=1$;
  \item  $e=2^{k_1+1}+1$, $s=2^{k_1}+1$ and $l=2^{k_1}+1$, where $k_1$ is odd and $\gcd(k_1,m)=1$;
  \item  $e=2^{k_1+1}+1$, $s=2^{k_1}$ and $l=2^{k_1}-1$, where $k_1$ is even and $\gcd(k_1,m)=1$;
  \item  $e=2^{k_1+1}-1$, $s=2^{k_1}$ and $l=2^{k_1}+1$, where $k_1$ is odd and $\gcd(k_1(k_1+1),m)=1$;
  \item  $e=2^{k_1+1}-1$, $s=2^{k_1}-1$ and $l=2^{k_1}-1$, where $k_1$ is even and $\gcd(k_1(k_1+1),m)=1$.
  \end{enumerate}
\end{prop}
\begin{proof} The above six cases can be proved in the same way, so we only show the assertion for the case (1) holds as follows.

 Note that the fact ${\rm gcd}(2^m+1,2^m-1)=1$ implies
$$\begin{array}{rcl}{\rm gcd}(d_1,2^{n}-1)&=&{\rm gcd}(d_1,2^m+1){\rm gcd}(d_1,2^m-1)
\\&=&{\rm gcd}(e-2s,2^m+1){\rm gcd}(e,2^m-1)
\\&=&{\rm gcd}(2^{k_2}-1,2^m+1){\rm gcd}(2^{k_3}+1,2^m-1).
\end{array}$$
Since $m$ and $k_2$  are odd, it follows from Lemma \ref{lem4} (ii) that $\gcd(d_1, 2^m-1)=1$. In addition,
by assumption that $m$ and $k_1$ are odd, $\gcd(k_1,m)=1$ and Lemma \ref{lem4} (iii), we have
$$r={\rm gcd}(2^{k_1}+1,2^m+1)=2^{(k_1,m)}+1=3,$$
and
$${\rm gcd}(e+l-2s,2^m+1)={\rm gcd}(2^{k_1}+2^{k_2},2^m+1)=\gcd(2^{k_2-k_1}+1,2^m+1)=1.$$
From the above analysis, the integers $e, s, l$ satisfy all conditions in Theorem \ref{thm1}, so the desired conclusion follows.
\end{proof}

In Propositions \ref{prop1}, the exponents $d_i\not\equiv 1\,({\rm mod}\,2^m-1)$ for $i\in\{1,2\}$. In the sequel, we consider the binomials with 
$d_1\equiv 1\,({\rm mod} \,2^m-1)$ and $d_2=1$, i.e., $d_1$ is a Niho exponent \cite{Niho}. By Theorem \ref{thm1}, we can obtain the following proposition.

\begin{prop}\label{prop3}
Let positive integers $n$, $m$, $s$ and $d_1$ satisfy  $n=2m$, $d_{1}=s(2^m-1)+1$, and  ${\rm gcd}(d_{1},2^n-1)=1$. Then
the binomial $$x^{d_{1}}+ux$$ is a permutation polynomial over $\mathbb{F}_{2^n}$
if the following conditions are satisfied:

(i) $r:={\rm gcd}(s,2^m+1)>1$;

(ii) ${\rm gcd}(s-1,2^m+1)=1$; and

 (iii) $u\in U\setminus U^r$.
\end{prop}

%A polynomial $f(x)\in \mathbb{F}_{q}[x]$ is {\it a complete permutation polynomial} (CPP)
% if both $f(x)$ and $f(x)+x$ are permutations of $\mathbb{F}_{q}$ \cite{Niederreiter}.
 For a positive integer $d$ and $\alpha\in
\mathbb{F}_{q}^{*}$, the monomial $\alpha x^d$ over $\mathbb{F}_q$ is
a CPP if and only if  $d$ satisfies ${\rm gcd}(d,q-1)=1$ and the
  binomial $\alpha x^d+x$ is a PP. Recently, some CPPs are presented in  \cite{PYuan1, TZH}.

By Proposition \ref{prop3} and the definition of CPPs, we obtain a class of monomial CPPs.
\begin{cor}\label{cor2}
Let positive integers $n$, $m$, $s$ and $d_1$ satisfy  $n=2m$, $d_{1}=s(2^m-1)+1$, and ${\rm gcd}(d_{1},2^n-1)=1$.
If $r={\rm gcd}(s,2^m+1)>1$ and ${\rm gcd}(s-1,2^m+1)=1$, then
$$u^{-1}x^{d_{1}}$$
is a complete permutation polynomial over $\mathbb{F}_{2^n}$ for each $u\in U\setminus U^r $.
\end{cor}

To find some special classes of monomial CPPs satisfying the conditions in Corollary \ref{cor2}, we need the following lemma.
\begin{lem}\label{lem3}
For an odd prime $p$, denote by $r$ the order of the element $2$ in $\mathbb{Z}/p\mathbb{Z}$, where  $\mathbb{Z}/p\mathbb{Z}$ is the residue class ring of integers modulo $p$. Then

(i) if $r$ is odd, ${\rm gcd}(p,2^k+1)=1$ for every positive integer $k$;

(ii) if $r$ is even, ${\rm gcd}(p,2^k+1)=1$ if $\frac{r}{2}\nmid k$, or $\frac{r}{2}\mid k$ and $\frac{2k}{r}$ is even.

\end{lem}
\begin{proof} For positive integers $t$ and $r$, we have
$$2^{tr}+1=(2^r-1)\left(2^{(t-1)r}+2^{(t-2)r}+\cdots +2^r+1\right)+2.$$
Then ${\rm gcd}\left(2^{tr}+1,2^r-1\right)={\rm gcd}\left(2,2^r-1\right)=1$.

(i) For odd $r$, if an integer $k$ satisfies $p\,|\, 2^k+1$, we have $p\,|\,(2^k+1)(2^k-1)=2^{2k}-1$
and then $2^{2k}\equiv 1\,({\rm mod}\,p)$.
Note that the order of the element $2$ in $\mathbb{Z}/p\mathbb{Z}$ is $r$. Consequently, we have  $r\,|\,2k$ and then $r\,|\, k$. Thus, ${\rm gcd}(2^k+1,2^r-1)={\rm gcd}(2^{\frac{k}{r}r}+1,2^r-1)=1$, and then $p\,|\,{\rm gcd}(2^k+1,2^r-1)=1$, which is impossible. As a consequence,
${\rm gcd}(p,2^k+1)=1$ for every positive integer $k$.

(ii) For even $r$, suppose $p\,|\,2^k+1$ for some $k$. By a similar analysis as in (i), we have $r\,|\,2k$ and then $\frac{r}{2}\,|\,k$. Thus, in the case of $\frac{r}{2}\nmid k$, we have ${\rm gcd}(p,2^k+1)=1$. Now suppose that $\frac{r}{2}\,|\,k$ and $\frac{2k}{r}$ is even. Denote $k=t\cdot \frac{r}{2}$ for some $t$,
then $t=\frac{2k}{r}$ is even and $\frac{t}{2}$ is a positive integer. By ${\rm gcd}(2^k+1,2^r-1)={\rm gcd}(2^{\frac{t}{2}\cdot r}+1,2^r-1)=1$ and $p\,|\, 2^r-1$, we have $p\nmid 2^k+1$, i.e. ${\rm gcd}(p,2^k+1)=1$. \end{proof}

With Lemma \ref{lem3}, we obtain six classes of CPPs as follows.

\begin{prop} Let positive integers $n$, $m$, $s$, and $d$ satisfy  $n=2m$, $d=s(2^m-1)+1$. Then
$$u^{-1}x^{d}$$
is a complete permutation polynomial over $\mathbb{F}_{2^n}$ if  one of the following conditions is satisfied:
\begin{enumerate}
 \item $s=2^k+1$, $m$ and $k$ are odd with $t=\gcd(m,k)$, and $u\in U\setminus U^{2^t+1}$;
\item $s=2^k+1$, $m$ and $k$ are even with $t=\gcd(m,k)$, $\frac{k}{t}$ and $\frac{m}{t}$ are odd, and $u\in U\setminus U^{2^t+1}$;
  \item $s=6$, $m$ is odd and $5\nmid m$, $u\in U\setminus  U^3$;\label{case3}
  \item $s=15$, $m$ is odd, and $u\in U\setminus  U^3$;
  \item $s=63$, $m$ is odd and $u\in U\setminus  U^3$;
  \item $s=2^m-2$, $m$ is odd and $u\in U\setminus  U^3$.
\end{enumerate}\end{prop}
\begin{proof} The first two cases can be proved by employing Corollary \ref{cor2} and Lemma \ref{lem4} in a similar way as in the proof of Proposition \ref{prop1}, whilst Lemma \ref{lem3} is necessary for the proofs of the other four cases.
Here we only give the proof of Case \ref{case3}: $s=6$, $m$ is odd, $5\nmid m$ and $u\in U\setminus  U^3$. Other cases can be similarly proved and their proofs are omitted.

Since $d=6(2^m-1)+1$, by Corollary \ref{cor2}, one needs to  check
whether the following conditions are satisfied:

(i) ${\rm gcd}(d,2^n-1)=1$;

(ii) ${\rm gcd}(6,2^m+1)> 1$; and

(iii) ${\rm gcd}(6-1,2^m+1)=1$.

 Note that $$(d, 2^n-1)=(6(2^m-1)+1, 2^m+1)(6(2^m-1)+1,2^m-1)=(11, 2^m+1).$$
 Since the order of $2$ in $\mathbb{Z}/11\mathbb{Z}$ is $10$, it follows from Lemma \ref{lem3} (ii) that ${\rm gcd}(11,2^m+1)=1$ if $5\nmid m$, and the
 first condition is met. It is clear that $\gcd(6,2^m+1)=3$ since $m$ is odd. In addition, the order of $2$ in
  $\mathbb{Z}/5\mathbb{Z}$ is $4$, it follows again from Lemma \ref{lem3} (ii) that ${\rm gcd}(5,2^m+1)=1$ since $m$ is odd. The above three conditions are thus satisfied. The desired conclusion then follows from Corollary \ref{cor2}.
\end{proof}

\section{Concluding Remarks and Further Works}

More classes of binomial permutation polynomials having a form as in (\ref{form1}) would be obtained if new parameters
 satisfying the conditions in Theorem \ref{thm1} are found. For the case of $t\geq 3$, it is interesting to present a criterion similar to Theorem \ref{thm1}.

Some examples of trinomial permutation polynomials having a form as in (\ref{form1}) can be found. For instance, with the help of a computer,  a conjecture is presented as follows.

{\bf Conjecture 1}:
Let $n=2m$ for an odd integer $m$. Then both
$$f(x)=x^{2^{m}+4}+x^{2^{m+1}+3}+x^{2^{m+2}+1}$$
and $$g(x)=x^{2^m}+x^{2^{m+1}-1}+x^{2^{2m}-2^m+1}$$
are permutation polynomials over $\mathbb{F}_{2^n}$.

This conjecture has been validated for $m=3$, $5$, $7$ and $9$. The main difficulty to prove the conjecture lies in
determining the number of solutions to the equations in Equation (\ref{eq1}).

It is also interesting to investigate the permutation behavior of the $p$-ary version of the functions having a form as in  (\ref{form1}), where $p$ is an odd prime.
In this case, the polar coordinate representation does not work. However, every nonzero $x\in \mathbb{F}_{p^{2m}}$ can be represented uniquely as $x=\alpha^i\beta^j$,
where $\beta$ is a primitive element of $\mathbb{F}_{p^{2m}}$, $\alpha=\beta^{p^m+1}$, $i=0, 1,\cdots, p^m-2$ and
$j=0, 1,\cdots, p^m$.

\end{document}